\documentclass[preprint]{elsarticle}
\usepackage[margin=3cm]{geometry}

\usepackage{amssymb}
\usepackage{amsmath}
\usepackage{amsthm}
\usepackage{amsfonts}
\usepackage{graphicx}
\usepackage{float}
 
\newtheorem{thm}{Theorem}[section]
\newtheorem{dfn}[thm]{Definition}
\newtheorem{lma}[thm]{Lemma}

\newtheorem{remark}{Remark}

\def\Kz{\ensuremath{\mathcal{K}_0}}
\def\Ko{\ensuremath{\mathcal{K}_1}}

\newcommand{\cL}{L^2(\Sigma)}
\newcommand{\cH}{H^{-\frac{1}{2},-\frac{1}{4}}(\Sigma)}
\newcommand{\cX}{\mathcal X}

\newcommand{\IR}{\mathbb{R}}
\newcommand{\cT}{\mathcal{T}}

\journal{Computers \& Mathematics with Applications}
\usepackage[colorinlistoftodos,prependcaption,textsize=tiny]{todonotes}

\begin{document}

\begin{frontmatter}



 \title{Sparse Grid Approximation Spaces for Space-Time Boundary Integral Formulations of the Heat Equation}

 \author[label1]{Alexey Chernov}
 \author[label2]{Anne Reinarz}
 
 \address[label1]{Institut f\"ur Mathematik, 
Carl von Ossietzky Universit\"at Oldenburg, DE, alexey.chernov@uni-oldenburg.de}
 \address[label2]{Institut f\"ur Informatik, Technical University of Munich,  \underline{reinarz@in.tum.de}}

\begin{abstract}
%The heat equation is a simple model problem for more general parabolic problems.
The aim of this paper is to develop stable and accurate numerical schemes for
boundary integral formulations of the heat equation with Dirichlet 
%\todo[inline]{or Neumann [??? but we do not do Neumann]}%
 boundary conditions.
%\todo[inline]{The boundary reduction reduces the order of the numerical method
%to $O(h_x^{-(d-1)})$, assuming the solution of the linear system is performed in linear complexity.}
The accuracy of Galerkin discretisations for the resulting boundary integral formulations depends mainly on the choice of
discretisation space. We develop a-priori error analysis utilising a proof technique that involves
norm equivalences in hierarchical wavelet subspace decompositions.
We apply this to a full tensor product discretisation, showing improvements over existing results, 
particularly for discretisation spaces having low polynomial degrees. We then use the norm equivalences
to show that an anisotropic sparse grid discretisation yields even higher convergence rates. 
Finally, a simple adaptive scheme is proposed
to suggest an optimal shape for the sparse grid index sets.
\end{abstract}


\end{frontmatter}







%%
%% End of file `elsarticle-template-num.tex'.
 
 \section{Introduction}
 Solutions to the heat equation are needed in many applications in physics, engineering 
 and financial modeling \cite{wayland}, \cite{wilmott}. 
The primary application in three dimensions is modeling heat flow in an
isotropic medium. 
Higher dimensional applications appear in, e.g. the valuation of financial derivatives or in
image analysis and machine learning \cite{witkin}.
 Since the analytic solutions to these  problems are typically not known, efficient 
numerical methods for solving them are important.

Standard methods for solving parabolic boundary value problems, such as finite element methods,
approximate the solution using a variational formulation on a subdivision of the domain $\Omega$
combined with a low-order time stepping scheme \cite{thomee}. 
%I decided to cite this paper from the diss. instead as the diss. is only available in German
Efficient variants of this finite element approach using various space-time sparse grid discretisations
 can be found in e.g. \cite{oeltz}.
 In complicated spatial domains, or for unbounded domains,
 generating a mesh of the domain $\Omega$ can be extremely challenging. In contrast, the boundary
 integral formulation of the problem only requires a mesh of the boundary $\Gamma=\partial\Omega$.
In many applications only the boundary values of the solution or  its derivatives
 are the quantities of interest. This data can be obtained directly as a solution of the boundary integral formulation.

After the boundary reduction the resulting Galerkin boundary integral formulation of the heat equation
becomes coercive \cite{costabelheat}. Hence it
remains stable for any conforming space-time discretisation. In particular, it remains stable for
arbitrary choices of mesh size versus time step size. 
For problems with inhomogeneous stationary boundary conditions, this allows us to solve the equation much more
efficiently, as a small number of time steps is sufficient. 
In order to allow easy error analysis of these methods for
 different choices of discrete approximation spaces we restrict ourselves to
 a Galerkin discretisation of the problem.
 
 The main aim of this paper is to show that the boundary integral formulation provides a stable
 method to solve the heat equation for various choices of discretisation spaces. We will supplement
 the obtained theoretical bounds by numerical experiments. 
  %The computational gains of a sparse grid formulation of the boundary element method
 %are easy to lose with a naive implementation of the method. 
 %Compared to finite elements there are two main points to consider. Firstly, the 
 %integrals that need to evaluated are singular and cannot be solved by standard quadrature
 %formulas. And secondly, the stiffness matrices resulting from these schemes are
 %generally full matrices. This makes it impossible to solve the resulting linear system
 %in linear complexity without some method to sparsify the matrix. We show that
 %a correctly chosen wavelet discretisation in space leads to a sufficiently sparse matrix.
 The computational gains of the boundary integral formulation are easily lost with a naive
 implementation of the method, largely due to the resultant full stiffness matrices and singular
 integrals.  In a subsequent paper we will expand on how to resolve these implementational issues.
 
 The paper is organised as follows.
 In Section \ref{sec-1} we formulate the heat equation, introduce the boundary reduction and
 the anisotropic Sobolev spaces that the boundary integral operators are posed in. 
 
 In Section \ref{sec-fulltp} we will show that the classical convergence 
  results from \cite{costabelheat}
  can be improved for particular choices of polynomial degrees by using a wavelet decomposition
  of the discrete spaces. More precisely, this error analysis yields an improvement for 
  low polynomial degrees. These low order polynomials are of interest as they can be easily 
  implemented, even when using wavelet bases.
 
 In Section \ref{sec-sg}  we  improve on the error analysis of the  sparse grid
 approximations to this problem given in \cite{chernovschwabheat}
 and \cite{chernovschwabstoch} and verify those results numerically.
 We show that the combination technique (\cite{zengergriebel}, \cite{garckegriebel})
 gives an efficient and easily implemented approximation of the sparse grid space for this problem.
 
 Further, we show examples of index sets produced by a simple adaptive algorithm. 
 The adaptive implementation suggests that the approximation spaces constructed 
 in \cite{griebelknapek} give higher convergence rates for these problems
 if an appropriate anisotropic scaling  in time and space is chosen.

 %\todo{We have to cite the dissertation by Daniel Oeltz: http://hss.ulb.uni-bonn.de/2006/0827/0827.pdf Could you suggest an appropriate location for this citation?}
 
 \section{Problem Formulation}\label{sec-1}
 We start by formulating the domain heat equation.
Let $\Omega \subset \IR^d$, $d\geq 2$ be a bounded domain with a smooth 
boundary $\Gamma := \partial \Omega$. 
The extension of the following theory to non-smooth, e.g. polygonal domains, is possible 
\cite{noondiss}, \cite{costabelheat}. However, for simplicity we
restrict ourselves to smooth domains $\Gamma \subset C^\infty$. Further, let $n$ be the outer normal 
vector field of $\Gamma$.

With $T>0$ we denote a finite time horizon and with $\mathcal{I}:=(0,T)$ the time interval of interest.
 The domain heat equation is defined
on the space-time cylinder $Q:=\Omega\times \mathcal{I}$.

\begin{dfn}
 Given boundary data $g:\Gamma\times\mathcal{I}\rightarrow\IR$ 
 find the solution $u:Q\rightarrow \IR$ satisfying:
\begin{equation}\label{1-dir}
 \begin{split}
  (\partial_t-\Delta)u &= 0,~~~~~~~~~~\text{in }Q\\
  u &= 0, ~~~~~~~~~~\text{at }\Omega\times \{t=0\}\\
  Tu &= g, ~~~~~~~~~~\text{in }\Gamma\times\mathcal{I},
 \end{split}
\end{equation}
where $T$ is the trace operator $\gamma_0u = u|_\Gamma$ in the case of Dirichlet boundary conditions, or
the normal derivative of the solution on the boundary $\gamma_1u=\partial_nu|_\Gamma$
in the case of Neumann boundary conditions.
\end{dfn}

\subsection{Boundary reduction}
The boundary element method relies on finding a formulation of the problem 
(\ref{1-dir}) which is posed on the mantle of the space-time cylinder $\Omega\times\mathcal{I}$.
For this we require the fundamental solution of the heat equation,
which is
\begin{equation}\label{eq-fundamental}
 G(x,t) = \begin{cases}
(4\pi t)^{-d/2}e^{-|x|^2/4t} & t\geq0\\0&t<0,
\end{cases}
\end{equation}
for any spatial dimension $d\geq 1$ (see \cite{costabelheat}).

Then we can apply Green's second theorem to problem (\ref{1-dir}) with either
Dirichlet or Neumann boundary conditions, giving the following
representation of the solution of the heat equation
\begin{equation}\label{eq-green}
\begin{split}
 u(x,t) = \int_0^T\int_{\partial\Omega} &\left[ G(x-y,t-s)\frac{\partial }{\partial n_y} u(y,s) 
- \frac{\partial}{\partial n_y}G(x-y,t-s)u(y,s)\right]dyds,
\end{split}
\end{equation}
where $n_y$ is outward unit normal to $\partial \Omega$ at the point $y$.

The boundary element method then consists of finding the missing boundary data in
(\ref{eq-green}). That is, either the boundary flux
$\gamma_1 u$
for the Dirichlet problem
or the boundary values $\gamma_0u$ for the Neumann problem.
This formulation can be simplified using the following operators:
\begin{dfn}
 The single layer heat potential is defined as
\begin{displaymath}
  \Kz(\varphi)(x,t) := \int_\Sigma \varphi(y,s)G(x-y,t-s)dyds~~~~~~~~~(x,t)\in Q
\end{displaymath}

 The double layer heat potential is defined as
\begin{displaymath}
  \Ko(\psi)(x,t) :=    \int_\Sigma \psi(y,s)\frac{\partial}{\partial n_y}G(x-y,t-s)dyds~~~~~~~~~(x,t)\in Q.
\end{displaymath}
\end{dfn}

Using the trace operators $\gamma_0$ and $\gamma_1$,  the representation formula (\ref{eq-green}) can now be written
more concisely as 
\begin{equation}\label{rep-formula}
 u = \Kz(\gamma_1 u) - \Ko(\gamma_0 u),~~~~~~\text{in Q}.
\end{equation}

 Let $\varphi\in H^{\frac{1}{2},\frac{1}{4}}(\Sigma)$ and $\psi\in H^{-\frac{1}{2},-\frac{1}{4}}(\Sigma)$.
 \begin{dfn}
 The single layer operator $V$ is defined as
\begin{equation}
 V\psi:=\gamma_0 \Kz\psi.
\end{equation}
 and the double layer operator $K$ is defined as
\begin{equation}
K\varphi:=\gamma_0\left(\Ko\varphi\right)|_Q+\frac{1}{2}\varphi.
\end{equation}
\end{dfn}
 
An extensive analysis of these operators, including results on their well-posedness and
regularity can be found in \cite{costabelheat} and \cite{noondiss}. 
Using these operators we can formulate
two methods to find solutions of the boundary integral formulation of the heat equation.
For this we first need to introduce the anisotropic Sobolev spaces these methods are posed in. 
Let $r,s \geq 0$
then
\begin{equation*}
\begin{split}
H^{r,s}(\IR^d \times \IR) &:= L^2(\IR; H^r(\IR^d)) \cap H^s(\IR; L^2(\IR^d)),
\end{split}
\end{equation*}
Their dual spaces are given by $H^{-r,-s} := (H^{r,s})'$.

We restrict these spaces to the bounded domain $Q$:
\begin{displaymath}
 \begin{split}
H^{r,s}(Q) &:= \{u|_Q~:~u\in H^{r,s}(\IR^d \times \IR) \},
\\ 
H^{r,s}(\Sigma) &:= \{u|_\Sigma~:~u\in H^{r,s}(\IR^d \times \IR) \}.
\end{split}
\end{displaymath}

They are equipped with the graph norm:
\begin{equation*}
\begin{split}
\|u\|_{H^{r,s}(Q)} &= \|u\|_{L^2(\mathcal{I};H^r(D))} +\|u\|_{H^s(\mathcal{I};L^2(D))} \\
&\sim \|u\|_{H^r(D)\otimes L^2(\mathcal{I})} +\|u\|_{L^2(D)\otimes H^s(\mathcal{I})}
\end{split}
\end{equation*}
These spaces are well-defined for $(r,s) \in [-1,1] \times \IR$ if $\Gamma \in C^{0,1}$ 
and for all $r,s \in \IR$ if $\Gamma \in C^\infty$ \cite{costabelheat}.

The Sobolev spaces with zero initial conditions at $t=0$ will be denoted by
\begin{displaymath}
 \begin{split}
\tilde H^{r,s}(Q) &:= \{u|_Q~:~u\in H^{r,s}((-\infty,T) \times \Omega),~u(x,t)=0,~t<0 \},
\\ 
\tilde H^{r,s}(\Sigma) &:= \{u|_\Sigma~:~u\in H^{r,s}((-\infty,T) \times \Gamma),~u(x,t)=0,~t<0  \}.
\end{split}
\end{displaymath}

For the sake of simplicity we will restrict ourselves in the following to  Dirichlet boundary conditions $Tu = \gamma_0 u$ and vanishing source terms $f = 0$.

We are now ready to give the two formulations of the boundary element method we will
work with. The direct method first calculates the boundary flux of the solution and then uses the 
representation formula to find the solution itself. This is particularly useful in engineering
applications, where the boundary flux is itself a quantity of interest \cite{costabelheat}.

\textbf{The direct method: }
\begin{enumerate}
 \item Find  $\psi\in \cH$ such that:
\begin{equation}
 V\psi = \left(\frac{1}{2}I+K\right)g.
\end{equation}
\item Then the unique solution $u\in \tilde H^{1,\frac{1}{2}}(Q)$ of the Dirichlet problem \eqref{1-dir} with $f=0$ can be represented by
\begin{equation}
u = \Kz\psi -\Ko g.
\end{equation}
\end{enumerate}

The second method we introduce is called the indirect method.  The main advantage of this method is that 
only the single layer operator $V$ needs to be assembled. This method could be used 
if the boundary flux is of no particular interest since the intermediate solution $\psi$ has no particular physical
meaning \cite{costabelheat}.

\textbf{The indirect method: }
\begin{enumerate}
 \item Find  $\psi\in H^{-\frac{1}{2},-\frac{1}{4}}(\Sigma)$ such that:
\begin{equation}
 V\psi = g.
\end{equation}
\item Then the unique solution $u\in \tilde H^{1,\frac{1}{2}}(Q)$ of the Dirichlet problem \eqref{1-dir} with $f=0$ can be represented by
\begin{equation}
u = \Kz\psi.
\end{equation}
\end{enumerate}
 
\subsection{Galerkin discretisation}
We now discretise these operators in time and space. In the following
sections we will discuss several different choices of discrete spaces, so we initially
keep the notation of this section general.
Let $\cX_L$ be a nested sequence of discrete spaces dense in $\cH$, i.e. 
$$\cX_0\subset \cX_1\subset ...\subset \cX_L\subset ... \subset \cH.$$
Then, the direct and indirect formulations of the heat equation with Dirichlet data
in the variational form are:
\begin{equation}\label{eq-disform}
\begin{split}
 &\text{Find }\psi_L\in \cX_L \text{ such that}\\
 &~~~~~~~       \langle V \psi_L, \eta \rangle = \langle g,\eta                \rangle, ~~~~\,~~~~~~~~~
    \text{for all }\eta\in \cX_L ~~ \text{ (Indirect method)}\\
 &\text{or }~~~ \langle V \psi_L, \eta \rangle = \langle (\frac{1}{2}+K)g,\eta \rangle, ~~~          
    \text{for all }\eta\in \cX_L ~~ \text{ (Direct method)}
\end{split}
\end{equation}

 The following lemma follows directly from the Lax-Milgram Lemma and the Lemma of C\'ea respectively,
using the coercivity and continuity of $V$ in the appropriate anisotropic Sobolev spaces
described above. See \cite{costabelheat}  for further details.

\begin{lma}\label{lma-coercive}
The solution $\psi_L\in\cX_L$ of the problems (\ref{eq-disform}) is unique and quasi-optimal:
\begin{equation}
 \|\psi-\psi_L\|_{\cH} \leq \frac{\|V\|}{c_v}\underset{\eta_L\in \cX_L}{\inf}\|\psi-\eta_L\|_{\cH}.
\end{equation}
Here $\|V\|$ and $c_v$ are the continuity and coercivity constants of the single layer
heat operator $V$ respectively.
\end{lma}

Let $\{b_i(x,t)\}_{i=1...N}=\{b_{i_x}(x) \cdot b_{i_t}(t)\}_{i=(i_x,i_t)}$ be a basis of
the discrete space $\cX_L$,
where $N$ is the number of basis functions.
Then, finding the discrete solution $\psi = \sum_{j=1}^N c_j b_j \in \cX_L$ 
requires the solution of a linear system of the form
$$\sum_{j=1}^N \langle V b_j, b_i\rangle c_j = \langle g, b_i\rangle,~~~~\forall i = {1...N}.$$

We refer to the resultant matrix as
%\todo{please check the order of the indices $i,j$ here and in what follows. 
%The heat kernel is not symmetric in $t$!}%
$A_{i,j} = \langle V b_j, b_i\rangle$.
The structure of the matrix $A$ follows directly from the form of the fundamental solution.
In particular, since $G(x,t) = 0$ if $t<0$, the
matrix is block lower triangular.

 \section{Error Analysis for Full Tensor-products}\label{sec-fulltp}
 We start with a simple discretisation of the integral formulation of the heat equation using
 tensor products of piecewise polynomial basis functions in time and space. In this section
 we give improved error bounds for certain choices of polynomial degrees. 
 The classical theory due to \cite{costabelheat} uses an Aubin-Nitsche duality argument and properties
 of the $L^2$-orthogonal projection. We propose a new proof using as a main ingredient norm equivalences,
 which can be shown using  a wavelet
 decomposition of the discrete spaces.
 %This method will also be used to derive error bounds
 %for the sparse grid discretisation in the following sections.
 
We start by introducing the spaces of piecewise polynomials in time and space. Let $\cX_i^x $ and $\cX_j^t $ be the
nested finite element spaces:
 \begin{equation*}
 \begin{split}
 &\cX_0^x \subset \cX_1^x \subset .... \subset \cX^x_{i}\subset H^{-\frac{1}{2}}(\Gamma)\text{ and }\\
 &\cX_0^t \subset \cX_1^t \subset .... \subset \cX^t_{j}\subset H^{-\frac{1}{4}}(\mathcal{I}).
 \end{split}
 \end{equation*}
 In time we simply define the discrete spaces as discontinuous piecewise
 polynomials on an equidistant grid, which is refined by bisection. More precisely,
 $$\cX_{j}^t=\{ v\in L^2(\mathcal{I})~:~v|_{[t_k,t_{k+1}]} \in \mathcal{P}_{p_t},~p_t\geq 0,~k=0,\dots,2^j-1\}.$$
 
In space we define the discrete spaces using parameterisations $\gamma:[0,1]^{d-1} \rightarrow \Gamma$  of the boundary of the spatial domain.
In two dimensions the smooth boundary $\Gamma$ can be paramaterised using a single smooth, $1$-periodic function $\gamma$. In higher dimensions or in 
the case of a piecewise smooth boundary $\Gamma$ the domain is cut up into non-overlapping patches and each is parameterised by a smooth function.
Thus, 
 \begin{displaymath}
  \cX_{i}^x=\{v\in L^2(\Gamma)~:~v|_\tau \circ \gamma\in\mathcal{P}_{p_x},~p_x\geq 0~\forall \tau\in\tau_i\},
 \end{displaymath}
where $\tau_i$ is the triangulation of the boundary at level $i$ and
$\mathcal{P}_{p_x}$ is the space
of polynomials of degree $p_x$. 
%\todo{What does it mean for $d \geq 2$? What are the elements in this case? Do you mean that $h_\ell = 2^{-\ell}$? We should modify this sentence.}%
Again, we refine by halving the support of the basis functions in each direction. 
For a two dimensional problem this is simply a bisection. For a three dimensional problem
discretised using triangles this corresponds to a red refinement.

Then the full tensor product space $\cX_L$ can be defined as $$\cX_L = \cX_L^x \oplus \cX_L^t.$$
 
Next we introduce a multilevel decomposition of the space $\cX_L$. This decomposition is needed to
show the norm equivalences that are central to the proof.
We construct a system of subspaces $W_j$, that
 are pairwise orthogonal with respect to the $L^2$ inner product. They give a multilevel decomposition of $\cX^x_i$ as 
 follows:
 \begin{equation}\label{eq-multilevel}
   \cX^x_i = W^x_0 \oplus ... \oplus W^x_i,~~~~~ \cX^t_j = W^t_0 \oplus ... \oplus W^t_j
 \end{equation}
 
 In particular, for  $\psi\in H^{r,s}(\Sigma)$ with 
 $\psi=\sum_{(\ell_x,\ell_t)\geq 0} w_{(\ell_x,\ell_t)}$ and $w_{(i,j)}\in W_{i}^x\otimes W_{j}^t$, we have (see \cite{griebelknapek})
 \begin{equation}\label{eq-13}
   \|\psi\|_{H^{r,s}(\Sigma)}^2 \sim \left\{
\begin{array}{cl}
\sum_{(\ell_x,\ell_t)\geq 0}2^{2\max\{r\ell_x,s\ell_t\}}\|w_{(\ell_x,\ell_t)}\|^2_{L^2(\Sigma)}, & r,s \geq 0, \\[1.5ex]
\sum_{(\ell_x,\ell_t)\geq 0}2^{-2\max\{|r|l_x,|s|l_t\}}\|w_{(\ell_x,\ell_t)}\|^2_{L^2(\Sigma)}, & r,s < 0.
\end{array} \right.
 \end{equation}
 
 These norm equivalences deliver (sharp) upper and lower bounds for the error estimates,
 which we use in the forthcoming error analysis.

 We generalise the tensor product spaces of the form $\cX_L^x\otimes\cX^t_L$ by using the above
 multilevel decomposition and including $W^x_i\otimes W^t_j$ for certain indices $(i,j)$.
 The resulting spaces are then described fully by the indices of the included
 subspaces. The index set corresponding to the anisotropic full tensor product space is
\begin{equation}
 I_L^\sigma= \{(\ell_x,\ell_t)~:~\ell_x\leq L/\sigma,~\ell_t\leq\sigma L\},
\end{equation}
 where $\sigma$ is a positive real parameter that can be chosen freely. The form of scaling used here
 mirrors the definition of anisotropic sparse grid index sets from \cite{griebelharbrecht} and allows
 us to easily compare to those sets in Section \ref{sec-sg}. The index set is visualised in Figure \ref{fig-fulltp}. In subsequent sections we will characterise
  various sparse grid discretisation spaces by other choices of index set.
  
 The parameter $\sigma$ denotes the scaling in time and space.
 For example, if $\sigma=2$, we use four times as many discretisation levels in time than in space. 
On the other hand if $\sigma<1$ we refine more strongly in space than in time.

 %This means
 %\begin{displaymath}
 % (\ell_x,\ell_t)\notin I_L^\sigma \Leftrightarrow (\ell_x,\ell_t)\in\{\ell_x\geq \lfloor L/\sigma\rfloor+1\}\cup \{\ell_t\geq\left\lfloor \sigma L \right\rfloor +1\}. 
 %\end{displaymath}
 %Switched to psi and psi_L here for consistency
 Recall that $\psi=\sum_{(\ell_x,\ell_t)\geq 0}w_{(\ell_x,\ell_t)}$. Then by letting
 $\psi_L= \sum_{(\ell_x,\ell_t)\in I_L^\sigma}w_{(\ell_x,\ell_t)}\in \cX_L$ we get the following by applying
 \eqref{eq-13} twice.
 
\begin{figure}[t]
  \centering
  \includegraphics{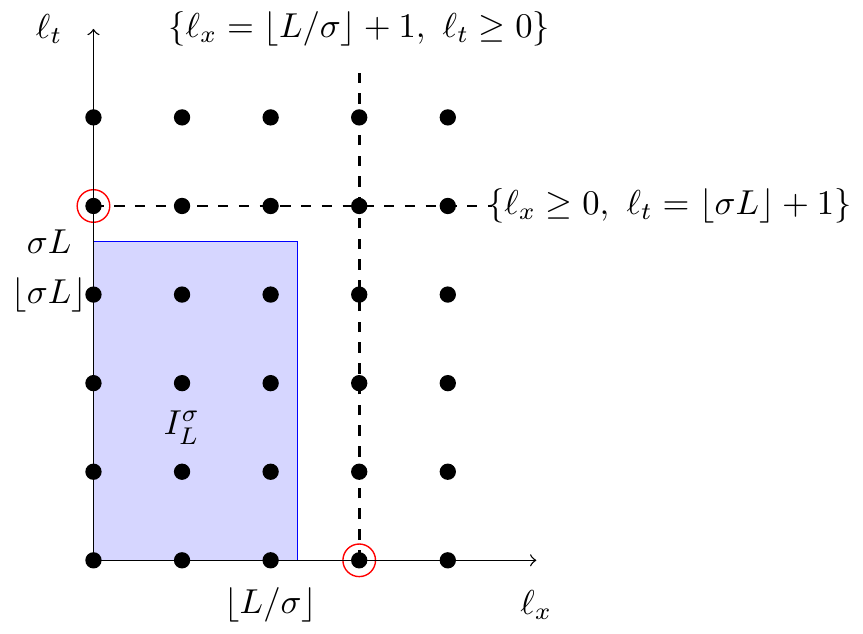}
 \caption{The full tensor product index set $I_{L}^\sigma$. The two potential minima outside the set
           $I_{L}^\sigma$ are circled in red.}
 \label{fig-fulltp}
\end{figure}
  
  \begin{equation}\label{eq-14}
  \begin{split}  
   &\inf_{\psi_L\in \cX_L}\|\psi-\psi_L\|^2_{\cH} \sim \sum_{(\ell_x,\ell_t)\notin I_L^\sigma}2^{-2\max\{\frac{\ell_x}{2},\frac{\ell_t}{4}\}}\|w_{(\ell_x,\ell_t)}\|^2_{\cL}\\ 
   &\leq \left(\max_{(\ell_x,\ell_t)\notin I_L^\sigma} 2^{-2\max\{\frac{\ell_x}{2},\frac{\ell_t}{4}\}-2\max\{\mu\ell_x,\lambda\ell_t\}}\right)
   \cdot \sum_{(\ell_x,\ell_t)\notin I_L^\sigma}2^{2\max\{\mu\ell_x,\lambda\ell_t\}}\|w_{(\ell_x,\ell_t)}\|^2_{\cL} \\
   & \leq \left(\max_{(\ell_x,\ell_t)\notin I_L^\sigma} 2^{-2\max\{\frac{\ell_x}{2},\frac{\ell_t}{4}\}
   -2\max\{\mu\ell_x,\lambda\ell_t\}}\right) \cdot \|\psi\|^2_{H^{\mu,\lambda}(\Sigma)}.
  \end{split}
 \end{equation}
 
  The term $2^{-2(\max\{\frac{\ell_x}{2},\frac{\ell_t}{4}\}+2\max\{\mu\ell_x,\lambda\ell_t\})}$ reaches its maximum when the negative
  exponent is as small as possible. We define
  \begin{equation}\label{eq-f}
   f(\ell_x,\ell_t):= \max\left\{\frac{\ell_x}{2},\frac{\ell_t}{4}\right\}+\max\{\mu\ell_x,\lambda\ell_t\}.
  \end{equation}
  Then, in order to determine the error bound we find the minimum
  \begin{equation}\label{eq-n}
   n := \min_{(\ell_x,\ell_t)\notin I_L^\sigma} f(\ell_x,\ell_t).
  \end{equation}
  To do so we use the following properties of monotonically increasing functions. 
  
  \begin{dfn}\label{dfn-mon}
   The function $F(\ell_x,\ell_t)$ is called \emph{a monotonically increasing function} if
   \begin{displaymath}
    \begin{split}
     &F(\ell_x+k,\ell_t) \geq F(\ell_x,\ell_t),~~~~~~\forall k\geq 0 \quad \text{ and }\\
     &F(\ell_x,\ell_t+k) \geq F(\ell_x,\ell_t),~~~~~~\forall k\geq 0.
    \end{split}
   \end{displaymath}
  \end{dfn}
 
   \begin{lma}\label{lma-mon}
   Let $F$ be a monotonically increasing function. Then its minimum outside the set $I_L^\sigma$ is
   \begin{displaymath}
    \min_{(\ell_x,\ell_t)\notin I_L^\sigma} F(\ell_x,\ell_t) = 
    \min\{ F(\lfloor L/\sigma \rfloor +1,0), F(0,\left\lfloor \sigma L \right\rfloor +1) \}.
   \end{displaymath}
  \end{lma}
  
  \begin{proof}
  Let $\ell_x\geq L/\sigma+1$ Then there holds
  \begin{displaymath}
   F(\ell_x,\ell_t) \geq F(\lfloor L/\sigma\rfloor+1,\ell_t)
  \end{displaymath}
  by Definition \ref{dfn-mon}.
  Analogously if we let $\ell_t\geq \left\lfloor \sigma L \right\rfloor +1$, there holds
    \begin{displaymath}
   F(\ell_x,\ell_t) \geq F(\ell_x,\left\lfloor \sigma L \right\rfloor +1)
  \end{displaymath}
  Together this tells us that the minimising multi-index lies in the subset
  \begin{displaymath}
   \{(\ell_x,\ell_t)~:~\ell_x=\lfloor L/\sigma \rfloor+1 \text{ or }l_t= \left\lfloor \sigma L \right\rfloor +1\}\subset \{(\ell_x,\ell_t)\notin I_L^\sigma\}.
  \end{displaymath}
  In Figure \ref{fig-fulltp} this subset is depicted by the dashed lines.
  
  Now let $\ell_x=\lfloor L/\sigma \rfloor+1$ and $\ell_t\geq0$, then there holds
  \begin{displaymath}
   F(\ell_x,\ell_t)\geq F(\lfloor L/\sigma \rfloor+1,0).
  \end{displaymath}
  Analogously, for $\ell_x\geq0$ and $\ell_t=\left\lfloor \sigma L \right\rfloor +1$ we have
  \begin{displaymath}
   F(\ell_x,\ell_t)\geq F(0,\left\lfloor \sigma L \right\rfloor +1).   
  \end{displaymath}
  This shows that the minimum can only be attained at
  $(\lfloor L/\sigma \rfloor + 1,0)$ or $(0,\left\lfloor \sigma L \right\rfloor +1)$ as
  asserted.
  \end{proof}

   To estimate the convergence rate from (\ref{eq-14}) we require the minimum $n$ defined in (\ref{eq-n}).
  Clearly, the function of the exponent $f(\ell_x,\ell_t)$ is a monotonically increasing function.
  Using Lemma \ref{lma-mon} we get
  \begin{displaymath}
   \begin{split}
    n &= \min_{(\ell_x,\ell_t)\notin I_L^\sigma}f(\ell_x,\ell_t) = \min\left\{\left(\mu+\frac{1}{2}\right)\left(\lfloor L/\sigma \rfloor+1\right),\left(\lambda+\frac{1}{4}\right)\left(\left\lfloor \sigma L \right\rfloor +1\right)\right\}\\
      &\sim \min\left\{\left(\mu+\frac{1}{2}\right)\sigma^{-1}, \left(\lambda + \frac{1}{4}\right) \sigma\right\}(L+1).
   \end{split}
  \end{displaymath}
  Thus, the minimum is
  \begin{equation}\label{eq-gensig}
   n\sim(L+1)\begin{cases}
   \left(\lambda + \dfrac{1}{4}\right) \sigma, & \sigma^2 \leq \dfrac{\mu+1/2}{\lambda+1/4} \\[2ex]
        \left(\mu+\dfrac{1}{2}\right)\sigma^{-1}, & \text{else}.
      \end{cases}
  \end{equation}
 
 We can now formulate the following theorem on the convergence in the energy norm.
 \begin{thm}\label{thm}
 Let $d>1$ and let $\mu,\lambda$ fulfil 
 $\mu\leq p_x+1$ and $\lambda\leq p_t+1$ and
 let $c>0$ be a constant depending only
 on $\sigma,d,\mu$ and $\lambda$. Further, denote by $N_L$ the total number of
 degrees of freedom, i.e. $N_L\sim 2^{(\sigma + (d-1)/\sigma)L}$. Then the convergence in the energy norm is  
   \begin{displaymath}
  \|\psi-\psi_L\|_{\cH}^2 \leq cN_L^{-\frac{2\mu+1}{d-1+\sigma^2}}\|\psi\|_{H^{\mu,\lambda}(\Sigma)}^2 
  \quad \text{for} \quad \sigma^2 = \frac{\mu+1/2}{\lambda+1/4}.
 \end{displaymath}
 \end{thm}
 
 \begin{proof}
 Inserting the exponent $n$ from (\ref{eq-gensig}) into (\ref{eq-14}) we get
 \begin{displaymath}
  \|\psi-\psi_L\|_{\cH}^2 \leq c2^{-2n}\|\psi\|_{H^{\mu,\lambda}(\Sigma)}^2 
 \end{displaymath}
  For $\sigma^2 = \frac{\mu+1/2}{\lambda+1/4}$ we have 
 $n = \left(\lambda + \dfrac{1}{4}\right) \sigma = \left(\mu+\dfrac{1}{2}\right)\sigma^{-1}$.
 Finally, rewriting in terms of degrees of freedom $N_L\sim 2^{(\sigma  + (d-1)/\sigma)L}$ we get
 \begin{displaymath}
 \begin{split}
  \|\psi-\psi_L\|_{\cH}^2 %&\leq cN_L^{-\frac{2\mu+1}{d-1+\sigma^2}\frac{L+1}{L}}\|\psi\|_{H^{\mu,\lambda}(\Sigma)}^2 
  %\leq c 2^{\sigma+(d-1)/\sigma}  N_L^{-\frac{2\mu+1}{d-1+\sigma^2}}\|\psi\|_{H^{\mu,\lambda}(\Sigma)}^2 \\
  &\leq cN_L^{-\frac{2\mu+1}{d-1+\sigma^2}}\|\psi\|_{H^{\mu,\lambda}(\Sigma)}^2 .
 \end{split}
 \end{displaymath}
 \end{proof}
 
 \begin{remark}
  The highest possible rate of convergence is attained in the space $H^{p_x+1,p_t+1}(\Sigma)$, 
 so we restrict the regularity constants $\mu$ and $\lambda$ by the polynomial degrees $p_x+1$ and
 $p_t+1$ respectively \cite{griebelharbrecht}.
 \end{remark}
 
  \begin{table}[tb]
 \centering
 \begin{tabular}{c||c|c}
  \hline
   \multicolumn{3}{c}{} \\[-.7em]
  \multicolumn{3}{c}{Full tensor product, $d=2$}\\[.3em]
  \hline
     & & \\[-.7em]
  \small{$(p_x,p_t)$} & \small{conv. rate $\gamma$}  & \small{scaling $\sigma^2$} \\[.3em]
  \hline
     & & \\[-.7em]
 $(0,0)$     & $\frac{15}{22}\sim .7$           & $\frac{6}{5}$   \\[.5em]
 $(1,0)$     & $\frac{5}{6}\sim .8$             & $2$   \\[.5em]
 $(0,1)$     & $\frac{9}{10}\sim.9$             & $\frac{2}{3}$   \\[.5em]
 $(1,1)$     & $\frac{45}{38}\sim 1.2$          & $\frac{10}{9}$  \\[.5em]
\end{tabular}\hspace{1em}
 \begin{tabular}{c||c|c}
 \hline
  \multicolumn{3}{c}{} \\[-.7em]
 \multicolumn{3}{c}{Full tensor product, $d=3$}\\[.3em]
 \hline
    & & \\[-.7em]
 \small{$(p_x,p_t)$} & \small{conv. rate $\gamma$} &  \small{scaling $\sigma^2$} \\[.3em]
 \hline
    & & \\[-.7em]
 $(0,0)$     & $\frac{15}{32}\sim .46$          & $\frac{6}{5}$  \\[.5em]
 $(1,0)$     & $\frac{5}{8}\sim .62$            & $2$   \\[.5em]
 $(0,1)$     & $\frac{9}{16}\sim.56$            & $\frac{2}{3}$   \\[.5em]
 $(1,1)$     & $\frac{45}{56}\sim .8$           & $\frac{10}{9}$ \\[.5em]
\end{tabular}
 \caption{Improved convergence rates and optimal values of $\sigma$ for full
 product discretisations in $2$ and $3$ dimensions.}
 \label{table-1}
\end{table}
 
In Table \ref{table-1} we give the convergence rates and optimal
choices of $\sigma$ for different combinations of polynomial degrees.
In the case of an analytic solution the only constraints on $\mu$ and $\lambda$ lie in
the polynomial degree of the discretisation. Thus, we can choose $\mu = p_x+1$
and $\lambda=p_t+1$ to maximise
the convergence rate. Then, the optimal scaling $\sigma$ can be written
in terms of polynomial degree as
$\sigma^2=\frac{4p_x+6}{4p_t+5}$.
In two dimensions and for $p_x=p_t=0$ the optimal scaling is $\sigma = \frac{6}{5}$ and this gives a higher
convergence rate than that for $\sigma =1$ in the classical theory from \cite{costabelheat}, suggesting that this
scaling should be used instead.

As the polynomial degrees in time and space $p_x=p_t$ increase the improvement becomes smaller.
The results from Theorem \ref{thm} and Table \ref{table-1} give the largest
improvement for the case $p_x=p_t=0$. This happens since for large $p_x=p_t$ the term $\frac{4\mu+2}{4\mu+1}$ approaches $1$ and
thus the results approach the results given in \cite{costabelheat}.

\subsection{Numerical Experiments}
We now verify the expected convergence rates for different choices of scaling factor $\sigma^2$.
In these numerical experiments we set the time horizon to $T=4$, giving the time interval
$\mathcal{I}:=(0,4)$. We initially choose a unit circle $B_1(0)$ as the domain $\Omega$.
Thus, $Q:= B_1(0) \times \mathcal{I}$ is the space-time cylinder
with mantle $\Sigma = \partial B_1(0)\times \mathcal{I}$
and we solve the Dirichlet problem (\ref{1-dir}) with the right hand side
$g(r,\varphi,t)=t^2\cos(\varphi)$. The advantage of this test is that the exact solution
can be calculated to verify the numerical solution \cite{owndiss}. 

In Figure \ref{fig-cost2} we plot the convergence of the
boundary flux in the energy norm squared. For the curve corresponding
to the scaling $\sigma = 6/5$ the expected convergence rate from Theorem \ref{thm} is $\frac{15}{11}$.
As we can see the convergence rate is close to the predicated rate.
We also run tests with two other values of $\sigma$. 
When $\sigma=1$, we have  a slightly lower convergence rate, as predicted by (\ref{eq-gensig}).
Further, the constant for $h_t\sim h_x^{6/5}$ leads to a lower overall error.
The convergence rate in this case is expected to be $\frac{5}{4}$.
Lastly, when $\sigma=2$ we expect a slower convergence rate of $1$ from (\ref{eq-gensig}).
The numerical tests confirm these  rates.
  
Next we give some more challenging tests on an ellipse with the
semi-axes $a=0.8$ and $b=0.5$. For these tests
it is simpler to use the indirect method, as the exact solution is not
known. We use a value calculated with $\sim 3.3\cdot10^7$ degrees of freedom
as an approximation of the exact solution to calculate the error. For this test we use the more oscillatory
right hand side $g(\varphi,t)=t^2\cos(2\varphi)$.

In Figure \ref{fig-ellipse1} one can see that the correspondence to the expected rates is good for
$\sigma=2$, where we again expect a rate of exactly $1$.
At $\sigma=1$ the rate  should be $5/4=1.25$,
and is in fact somewhat higher than that.
In particular, the error for $\sigma=1$ is smaller
than the error for $\sigma=\frac{6}{5}$. The reason for this discrepancy is not clear.
The expected convergence rate for $\sigma=\frac{6}{5}$ is $\frac{15}{11}$, and Figure \ref{fig-ellipse1} shows a good
correspondence to this rate.

   \begin{figure}[tb]
  \includegraphics[scale=0.39]{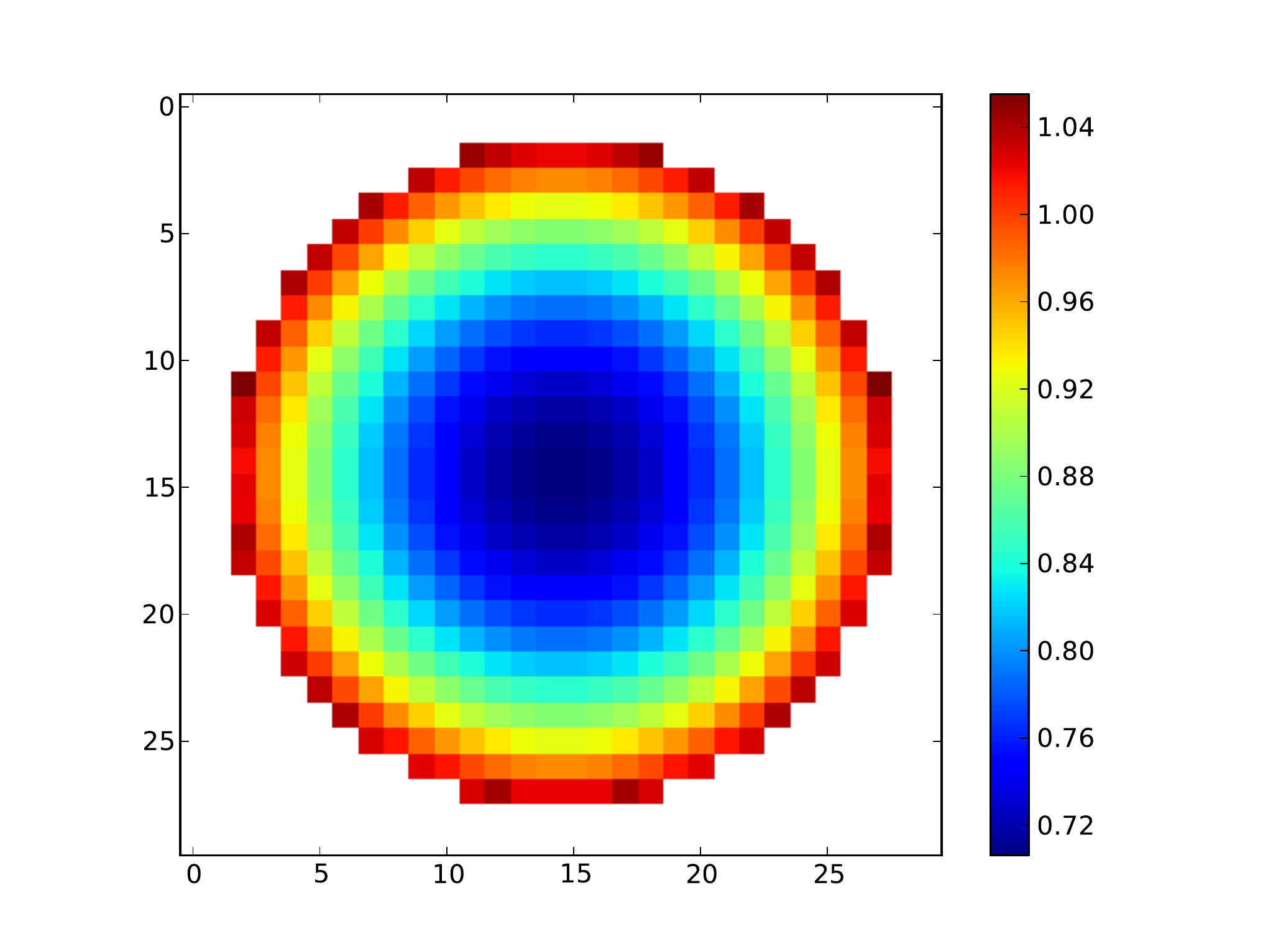} 
  \includegraphics{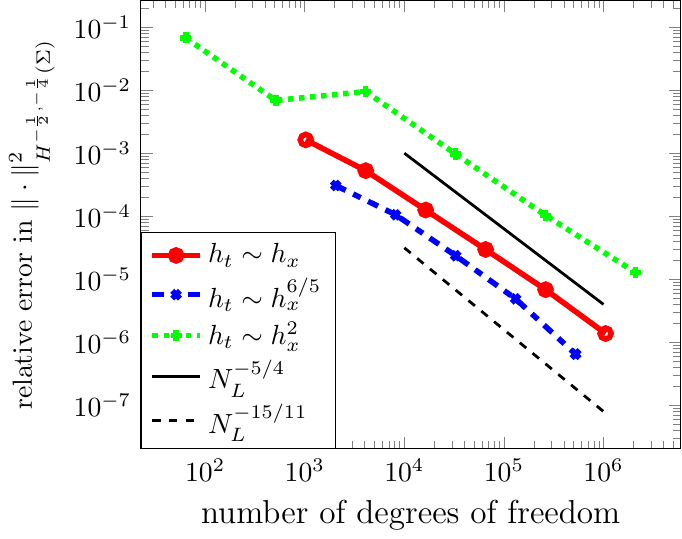}
 \caption{Convergence of the boundary flux in the energy norm for the right hand side
 $g(r,\varphi,t)=t^2\cos(\varphi)$ in the case $p_x=p_t=0$.}
 \label{fig-cost2}
 \end{figure}

  \begin{figure}[tb]
 \includegraphics[scale=0.39]{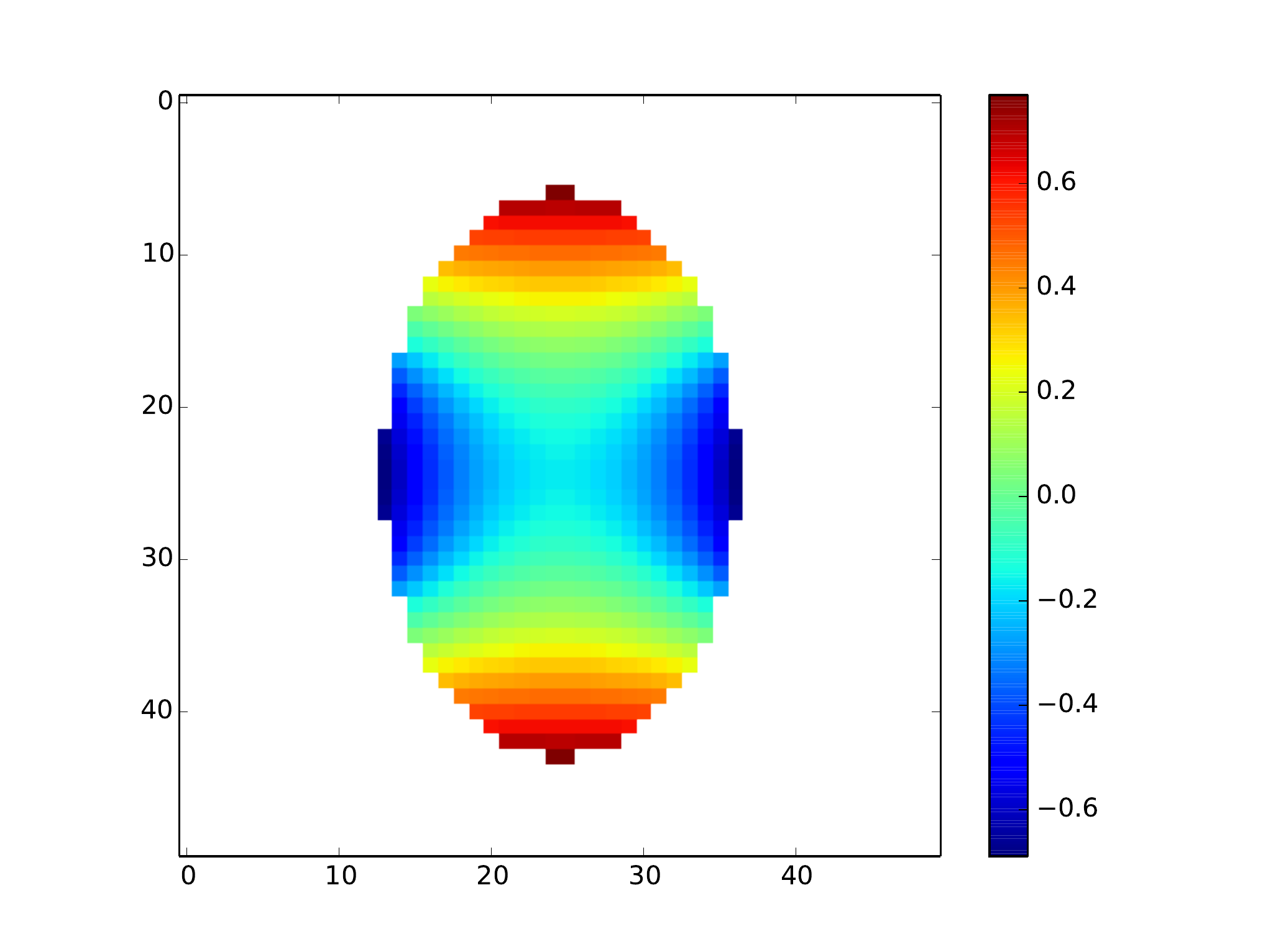}  
 \includegraphics{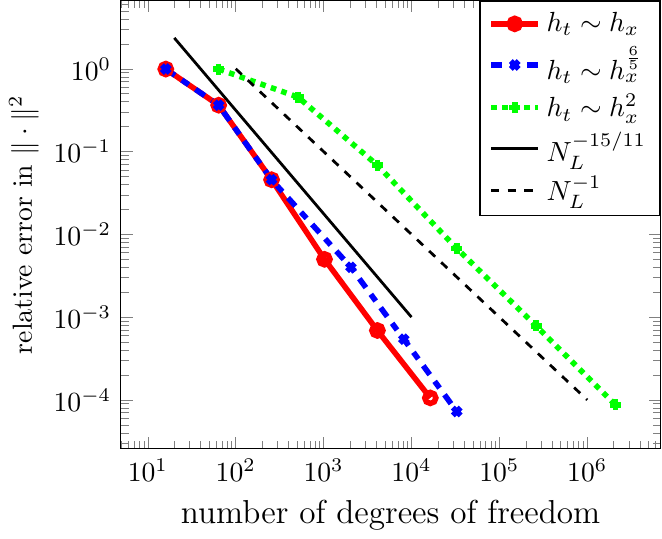}
 \caption{Left: The approximated solution on an ellipse for $g(\varphi,t)=t^2\cos(2\varphi)$  at timestep $t=1$.
 Right: Convergence of the boundary flux in  squares of the energy norm for the right hand side $g(\varphi,t)=t^2\cos(2\varphi)$ 
on an ellipse with eccentricities $a=0.8,b=0.5$.}
 \label{fig-ellipse1}
\end{figure}

 \section{Sparse Grids}\label{sec-sg}
The natural tensor product structure of $\Sigma= \Gamma\times\mathcal{I}$
in space and time lends itself to a sparse grid approximation. The main idea
behind these methods is to truncate the tensor-product expansion
of a one-dimensional multilevel basis to improve the cardinality of the tensor product in
space-time. To approximate a function such as $f(x,t)=f_1(x)f_2(t)$ that can  be
decomposed entirely in time and space the subspaces that are highly refined both in space
and time can be dropped entirely without a loss of approximation properties. This is particularly
helpful as these are precisely the most 'expensive' subspaces, i.e. those containing the most
degrees of freedom.

A sparse Galerkin discretisations yields only a mild dependence on the dimension.
More precisely, standard sparse grid methods
scale in dimension with $\mathcal{O}(N(\log N)^{d-1})$ \cite{zenger}, while the full tensor product scales
with $\mathcal{O}(N^d)$, where $N$ is the number of degrees of freedom. Thus, the sparse
grid approximation can alleviate the 'curse of dimensionality'.

 Let $\cX_L^x$ be the discrete space in the spatial dimensions and $\cX_L^t$ the
discrete space in time. The first step in defining sparse grid structures is the definition
of one-dimensional multilevel decompositions. They can then be combined
to form sparse grid spaces. As in (\ref{eq-multilevel}) assume that there exists a system 
of subspaces $W_i^x$ and $W_j^t$ such that
\begin{displaymath} 
\cX_L^x = W_0^x \oplus \cdots \oplus W_L^x,\qquad
\cX_L^t = W_0^t \oplus \cdots \oplus W_L^t. 
\end{displaymath}

The general form of a sparse grid space in two dimensions is
\begin{equation*}\label{SparseTP}
\hat \cX_L := \bigotimes_{(\ell_x,\ell_t)\in J_L}
W_{\ell_x}^x \otimes W_{\ell_t}^t
\subset \cX = H^{-\frac{1}{2},-\frac{1}{4}}(\Sigma)
\end{equation*}
where $J_L$ is an index set, which can be freely chosen.

The approximation properties of the space $\hat\cX_L$ are entirely dependent on the choice of index
set $J_L$. In Section \ref{sec-fulltp} for example we have seen the approximation properties
for the set index set $I_L^\sigma= \{(\ell_x,\ell_t)~:~\ell_x\leq L/\sigma,~\ell_t\leq\sigma L\}$,
there we essentially only had to find the minimum of a monotonically increasing function
outside of the subset $I_L$ in order to find the error in the energy norm. The same 
is true for more general index sets $J_L$. 

\begin{dfn}
The standard anisotropic sparse grid index set is defined as
\begin{equation}\label{sg-indset}
 \hat J_L^\sigma = \left\{ (\ell_x,\ell_t)~:~\ell_x\sigma + \ell_t/\sigma\leq L \right\},
\end{equation}
where $0<\sigma<\infty$ is a free parameter  (see e.g. \cite{griebelbungartz}). In this case we write $\hat\cX_L=\hat\cX_L^\sigma$.
\end{dfn}

\subsection{Error Analysis of Sparse Grids}
 An extensive error analysis of the standard sparse grid approximation to this problem
 has been shown in \cite{chernovschwabheat} and \cite{chernovschwabstoch} using an
 Aubin-Nitsche based argument.  
As in the error analysis for the full tensor product spaces in Section \ref{sec-fulltp} the main ingredient will
be norm equivalences, shown using wavelet bases.  For sparse grid approximation spaces the appropriate Sobolev spaces are spaces of weak mixed derivatives,
\begin{equation}
H^{r,s}_\text{mix}(\Sigma):= H^r(\Gamma) \otimes H^s(\mathcal{I}).
\end{equation}

For  $u\in H_\text{mix}^{r,s}(\Sigma)$ with 
$u=\sum_{(\ell_x,\ell_t)\geq 0} w_{(\ell_x,\ell_t)}$ and $w_{(\ell_x,\ell_t)}\in W_{\ell_x}\otimes W_{\ell_t}$, we have
 \begin{equation}\label{eq-mix}
  \|u\|_{H^{r,s}_\text{mix}(\Sigma)}^2\sim \sum_{\ell_x,\ell_t}2^{2(r\ell_x+s\ell_t)}\|w_{(\ell_x,\ell_t)}\|^2_{L^2(\Sigma)}.
 \end{equation}

Define the set of indices outside of the sparse grid index set $\hat J_L^\sigma$ as 
 $$I=\{(\ell_x,\ell_t)~:~\ell_x\sigma+\ell_t/\sigma>L\}.$$
 
 Then as in Section \ref{sec-fulltp} we use the norm equivalences (\ref{eq-mix}) to estimate
 \begin{displaymath}
  \begin{split}
   \|\psi-\psi_L\|_{\cH}^2 \leq\sum_{(\ell_x,\ell_t)\notin J_L^\sigma} 2^{-2\max\{\frac{\ell_x}{2},\frac{\ell_t}{4}\}
  -2(\mu\ell_x+\lambda \ell_t)}\|\psi\|^2_{H^{\mu,\lambda}_\text{mix}(\Sigma)}.
  \end{split}
  \end{displaymath}
  
According to \cite{griebelharbrecht} the index set $I$ can be split into two disjoint sets $I=I_1\cup I_2$, given by
 \begin{equation}\begin{split}\label{eq-splitting}
  &I_1 = \{(\ell_x,\ell_t)~:~0\leq \ell_x \leq \frac{L}{\sigma}, ~ L\sigma-\ell_x\sigma^2<\ell_t\}\\
\text{and }
  &I_2 = \{(\ell_x,\ell_t)~:~\frac{L}{\sigma}<\ell_x,~0\leq \ell_t\}.
 \end{split}\end{equation}

   We set $\tilde f(\ell_x,\ell_t):=\max\{\frac{\ell_x}{2},\frac{\ell_t}{4}\}
  +(\mu\ell_x+\lambda\ell_t)$ and apply the splitting (\ref{eq-splitting})
to estimate
   \begin{displaymath}
   \begin{split}
     \|\psi-\psi_L\|_{\cH}^2 
  &\leq
    \sum_{(\ell_x,\ell_t)\in I} 2^{-2\tilde f(\ell_x,\ell_t)}\|\psi\|^2_{H^{\mu,\lambda}_\text{mix}(\Sigma)}\\
   &= \left(\sum_{\ell_x=0}^{\lfloor L/\sigma\rfloor}\sum_{\ell_t=\lfloor L\sigma-\ell_x\sigma^2\rfloor +1}^\infty
   2^{-2\tilde f(\ell_x,\ell_t)}
  +\sum_{\ell_x=\lfloor L/\sigma\rfloor+1}^\infty\sum_{\ell_t=0}^\infty  
  2^{-2\tilde f(\ell_x,\ell_t)} \right)\|\psi\|^2_{H^{\mu,\lambda}_\text{mix}(\Sigma)}
  \end{split}
   \end{displaymath}
   
   The terms of the infinite sums can be rewritten using the following well-known property of
   the geometric series
   $\sum_{k=m}^\infty r^k=\frac{r^{m}}{1-r},~r<1$.  This gives
      \begin{displaymath}
   \begin{split}
     \|\psi-\psi_L\|_{\cH}^2 
  &\lesssim \Big(\sum_{\ell_x=0}^{\lfloor L/\sigma\rfloor } 
                        2^{-2\tilde f(\ell_x,\lfloor L\sigma-\ell_x\sigma^2\rfloor )}
   + \underbrace{\strut 
      2^{-2\tilde f(\lfloor L/\sigma\rfloor,0)}}_{\leq
      2^{-2(\mu+\frac{1}{2})\lfloor L/\sigma\rfloor}} \Big)\|\psi\|^2_{H^{\mu,\lambda}_\text{mix}(\Sigma)}
   \end{split}
   \end{displaymath}
   
To estimate the first summand the exponent $\tilde f(\ell_x,\lfloor L\sigma-\sigma^2\ell_x\rfloor)$ must be minimised
in the range $0\leq \ell_x\leq \lfloor  L/\sigma\rfloor $. This exponent is piecewise linear. Thus, its minima lie either at the limits
of the range $\ell_x=0$ and $\ell_x=\lfloor L/\sigma\rfloor$ or at the turning point of 
$\tilde f(\ell_x,\lfloor L\sigma-\sigma^2\ell_x\rfloor).$

Let $c>0$ be a constant that depends only on $\mu$ and $\lambda$.
To simplify the expression and to determine the turning point we use the estimate $m-1\leq \lfloor m \rfloor$  and get 
$$2^{-2\tilde f(\ell_x,\lfloor L\sigma-\sigma^2\ell_x\rfloor)}\leq 2^{2(\lambda+\frac{1}{4})}2^{-2\tilde f(\ell_x,L\sigma-\sigma^2 \ell_x)}\leq c 2^{-2\tilde f(\ell_x,L\sigma-\sigma^2 \ell_x)}.$$
Then the turning point of $\tilde f(\ell_x, L\sigma-\sigma^2\ell_x)$ lies at $$\ell_x = \frac{L\sigma-\sigma^2\ell_x}{2} 
\Leftrightarrow \ell_x = \frac{L\sigma}{2+\sigma^2}.$$
We examine these three potential minima separately:

\begin{itemize}
 \item $\ell_x=0,\,\ell_t=\lfloor L\sigma\rfloor$ gives the exponent: $\tilde f(0,\lfloor L\sigma\rfloor) = (\lambda+\frac{1}{4})\lfloor L\sigma\rfloor$. 
 In this case we can estimate $$2^{-2\tilde f(0,\lfloor L\sigma\rfloor)} = 2^{-2(\lambda+\frac{1}{4})\lfloor L\sigma\rfloor}\leq c 2^{-2(\lambda+\frac{1}{4}) L\sigma}.$$
 
 \item $\ell_x=\lfloor L/\sigma \rfloor,\ell_t = 0 $ gives the exponent $\tilde f(\lfloor L/\sigma \rfloor, 0) (\mu+\frac{1}{2})\lfloor L/\sigma\rfloor$. 
 In this case we can estimate $$2^{-2\tilde f(\lfloor L/\sigma \rfloor,0)} = 2^{-2(\mu+\frac{1}{2})\lfloor L/\sigma\rfloor}\leq c 2^{-2(\mu+\frac{1}{2}) L/\sigma}.$$
 
 \item $\ell_x=\frac{L\sigma}{2+\sigma^2},\,\ell_t = 2\ell_x$ gives the exponent $\tilde f(\ell_x,2\ell_x)  = \left(\mu+2\lambda+\frac{1}{2}\right)L\frac{\sigma}{2+\sigma^2}$.
 
\end{itemize}

Taking these terms together we get
\begin{equation}\label{eq-gamma}
 \|\psi - \psi_L\|_{\cH}^2 \leq c 2^{-2L\min\left\{\left(\mu+\frac{1}{2}\right)\sigma^{-1},
\left(\lambda+\frac{1}{4}\right)\sigma,\left(2\lambda+\mu+\frac{1}{2}\right)
\left(\frac{\sigma}{2+\sigma^2}\right)\right\}}\|\psi\|_{H^{\mu,\lambda}_\text{mix}(\Sigma)}.
\end{equation}

 Finally, we note that the dimension of the approximation space is given by  (see  \cite{griebelharbrecht}):
 \begin{equation}\label{eq-dof}
  \dim(\hat\cX_L^\sigma) =: N_L\sim\begin{cases}
          2^{L\sigma}L,&\text{ if } \sigma^2=d-1,\\
          2^{L\max\{(d-1)/\sigma,\sigma\}},&\text{ else.}
         \end{cases}
 \end{equation}

 In \cite{griebelharbrecht} the following three reasonable strategies for choosing the scaling factor $\sigma$ are suggested.
 \begin{enumerate}
  \item Equilibrating the degrees of freedom in the spaces $W_{\ell_x}^x\otimes W_{\ell_t}^t$ with 
  $\ell_x\sigma + \ell_t/\sigma = L$ leads to the choice $\sigma^2=d-1$.
  \item Equilibrating the approximation power of the tensor-product spaces leads to the choice $\sigma^2=\frac{\mu}{\lambda}$.
  \item Equilibrating a cost-benefit ratio leads to the choice $\sigma^2=\frac{d-1+\mu}{1+\lambda}$.
 \end{enumerate}
We refer to \cite{griebelharbrecht} for a closer discussion of these choices and note only that
in our implementation each degree of freedom in time and space requires roughly equal computational effort.
We  apply the first strategy, choosing $\sigma^2=d-1$.

\begin{remark}
 In two dimensions and for equal polynomial degrees all three strategies coincide.
\end{remark}

 \begin{thm}\label{thm-sg}
    Let $\psi$ and $\psi_L$ denote the exact solution to (\ref{eq-disform}) and the approximation
  in the discrete space $\cX_L$ with $\sigma^2 = d-1$ respectively.
 Further, let $d>1$ and let $\mu,\lambda$ fulfil $\mu\leq p_x+1$ and $\lambda\leq p_t+1$ and let
   $c> 0$ be a constant depending only on $\mu$ and $\lambda$. 
 Then the convergence in the energy norm is given by
   \begin{displaymath}\begin{split}
  \|\psi - \psi_L\|_{\cH}^2 &\leq c L^{2\gamma} N_L^{-2 \gamma}
  \|\psi\|_{H^{2\lambda+\mu,\lambda+\frac{\mu}{2}}(\Sigma)}^2 ,
\end{split}\end{displaymath}
where  \begin{equation}
\gamma = \min\left\{\frac{\mu+\frac{1}{2}}{\sigma^2},
\lambda+\frac{1}{4}, \frac{2\lambda+\mu+\frac{1}{2}}{2+\sigma^2}\right\}.
\end{equation}
  \end{thm}
  
 \begin{proof}
  Applying equation (\ref{eq-dof}) with $\sigma^2 = d-1$ to the estimate (\ref{eq-gamma}) directly gives 
     \begin{displaymath}\begin{split}
  \|\psi - \psi_L\|_{\cH}^2 &\leq c L^{2\gamma} N_L^{-2 \gamma}
  \|\psi\|_{H_\text{mix}^{\mu,\lambda}(\Sigma)}^2 .
\end{split}\end{displaymath}
This gives an estimate in the space $H_\text{mix}^{\mu,\lambda}(\Sigma)$. To get regularity in
the non-mix spaces we simply apply the following result from \cite{chernovschwabheat}.
Let $a,b,k\geq 0$ and $k\geq a+2b$, then
\begin{displaymath}
H^{k,\frac{k}{2}}(\Sigma)\subset H_\text{mix}^{a,b}(\Sigma).
\end{displaymath}
 \end{proof}

The convergence rates for different dimensions and choices of polynomial degrees (up to logarithmic terms)
are summarised in Table \ref{table-1sg} for $d=2$ and
in Table \ref{table-11} for $d=3$. We compare with the results for full tensor product discretisations given in Section \ref{sec-fulltp}.

The table shows that in discretisations with low polynomial degree the sparse grids yield higher rates
than the full tensor products. However, they require higher regularity assumptions on the data.

In the case of piecewise constant basis functions, i.e. $p_x=p_t=0$ and $d=2$, the convergence
rate in the energy norm using these sparse grids is almost twice as high as that of full tensor products.
For $d=3$ the improvements to the convergence rates $\gamma$ are not quite as large as in two dimensions
since we still discretise with a full tensor product in the spatial dimensions.
We note that for $d=3$ these rates give an improvement over
the rates given in \cite{chernovschwabheat}.
\begin{table}[H]
\centering
 \begin{tabular}{c||c|c}
  \hline
   \multicolumn{3}{c}{} \\[-.7em]
  \multicolumn{3}{c}{Full tensor product, $d=2$}\\[.3em]
  \hline
     & & \\[-.7em]
  $(p_x,p_t)$ & conv. rate $\gamma$ & scaling $\sigma^2$ \\[.3em]
  \hline
     & & \\[-.5em]
  $(0,0)$     & $\frac{15}{22}$          & $\frac{6}{5}$\\[.5em]
  $(1,0)$     & $\frac{5}{6}$            & $2$ \\[.5em]
  $(1,1)$     & $\frac{45}{38}$          & $\frac{10}{9}$ \\[.5em]
  $(3,1)$     & $\frac{3}{2}$            & $2$ \\[.5em]
 \end{tabular}\hspace{0.5em}
 \begin{tabular}{c||c|c}
  \hline
  \multicolumn{3}{c}{ } \\[-.7em]
  \multicolumn{3}{c}{Sparse grids, $d=2$}\\[.3em]
  \hline
  & & \\[-.7em]
  $(p_x,p_t)$ & conv. rate $\gamma$  & scaling $\sigma^2$ \\[.3em]
  \hline
  & & \\[-.7em]
  $(0,0)$     &   $\frac{7}{6}$      & $1$          \\[.5em]
  $(1,0)$     &   $\frac{5}{4}$      & $1$          \\[.5em]
  $(1,1)$     &   $\frac{13}{6}$      & $1$          \\[.5em]
  $(3,1)$     &   $\frac{9}{4}$      & $1$          \\[.5em]
 \end{tabular}
\caption{Convergence rates for full and sparse tensor product discretisation in $2$ dimensions.}
\label{table-1sg}
\end{table}
 \begin{table}[H] 
 \centering
  \begin{tabular}{c||c|c}
  \hline
   \multicolumn{3}{c}{} \\[-.7em]
  \multicolumn{3}{c}{Full tensor product, $d=3$}\\[.3em]
  \hline
     & & \\[-.5em]
  $(p_x,p_t)$ & conv. rate $\gamma$ & scaling $\sigma^2$ \\[.3em]
  \hline
     & & \\[-.7em]
  $(0,0)$     & $\frac{15}{32}$                     & $\frac{6}{5}$ \\[.5em]
  $(1,0)$     & $\frac{5}{8}$                       & $2$ \\[.5em]
  $(1,1)$     & $\frac{45}{56}$                     & $\frac{10}{9}$ \\[.5em]
  $(3,1)$     & $\frac{9}{8}$                       & $2$ \\[.5em]
 \end{tabular}\hspace{0.5em}
 \begin{tabular}{c||c|c}
  \hline
  \multicolumn{3}{c}{ } \\[-.7em]
  \multicolumn{3}{c}{Sparse grids, $d=3$}\\[.3em]
  \hline
  & & \\[-.7em]
  $(p_x,p_t)$ & conv. rate $\gamma$ & scaling $\sigma^2$ \\[.3em]
  \hline
  & & \\[-.7em]
  $(0,0)$     &   $\frac{3}{4}$      & $2$          \\[.5em]
  $(1,0)$     &   $\frac{9}{8}$      & $2$          \\[.5em]
  $(1,1)$     &   $\frac{5}{4}$      & $2$          \\[.5em]
  $(3,1)$     &   $\frac{17}{8}$      & $2$          \\[.5em]
 \end{tabular}
\caption{Convergence rates  for full and sparse tensor product discretisation in $3$ dimensions.}
\label{table-11}
\end{table}

\goodbreak

% For this choice of index set we know from \cite{chernovschwabheat} that the error in the energy
% norm is given as follows.
% \begin{thm}\label{thm-stsp}
% Suppose $\psi\in \tilde H^{(d-1)\mu,\mu}_\text{mix}(\Sigma)$ for $\mu,p_x,p_t$ satisfying
%  \begin{equation}\label{eq-pres}
%   \mu=\frac{p_x+1}{d-1},~~~\text{ and }~~~ p_t+1\geq\mu.
%  \end{equation}
%  
%  Then the error of the sparse tensor Galerkin approximation $\psi_L\in \hat\cX_L^\sigma$, 
%  with $\sigma=\sqrt{d-1}$ is
%  \begin{equation}
%   \|\psi-\psi_L\|_{\cH} \leq cN_L^{-\lambda}\left(\log(N_L)\right)^{\mu+\frac{1}{2}}\|\psi\|_{H^{(d-1)\mu,\mu}_\text{mix}(\Sigma)},
%  \end{equation}
% where $N_L$ is the number of degrees of freedom and
% $\lambda = \mu+\frac{1}{2(d+1)}$.
% \end{thm}

\subsection{Combination Technique}
The combination technique for the
solution of sparse grid problems was first introduced in \cite{zengergriebel}.
The basic idea behind the technique is to find a sparse grid approximation using a linear
combination of smaller full grid solutions. The advantage of this is that
 the included full grids are much smaller than the full sparse grid and can
 be computed more quickly, while still giving the same accuracy.
 It also gives an easier implementation since the
 need for the solution in a sparse grid
space is replaced with the solution of several full grids.
Further, the solution of the systems corresponding to these full grids can be performed in parallel,
see e.g. \cite{griebelparallel1} and \cite{griebelparallel2}.

No general proof of convergence for the combination technique exists. However,
it has been shown in \cite{griebelharbrecht2} that it produces the same
order of convergence with the same complexity as the Galerkin approach 
in the standard sparse tensor product case for elliptic operators acting on arbitrary Gelfand triples.
In particular, we can apply \cite[Theorem 2]{griebelharbrecht2} to our problem since $\mathcal{I}$ and
$\Gamma$ are bounded and the bilinear form (\ref{eq-disform}) is continuous and coercive \cite{costabelheat}
and thus guarantee convergence of the combination technique.

\begin{figure}[tb]
 \centering
\includegraphics{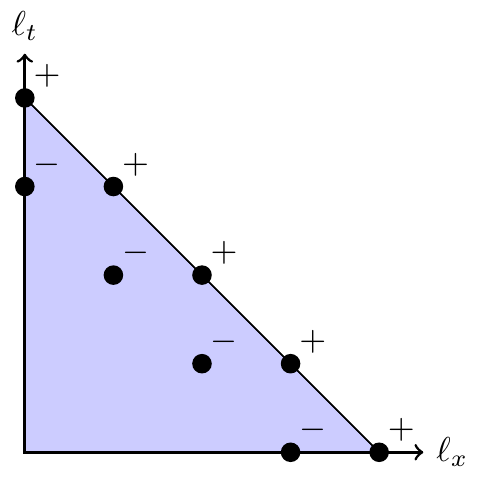}
\caption{The sign contributions of the subspaces used for the 
combination technique for standard sparse grids with $\sigma=1$.}
\label{fig-combinationindex}
\end{figure}

To introduce the combination  technique we use the Galerkin projection onto
the full tensor product discrete spaces.
The Galerkin projection is well-defined due to the coercivity of the single-layer operator $V$.
\begin{dfn}
Let $\Pi_{\cX}$ be a mapping
 \begin{displaymath}
  \Pi_{\cX}\,:\,\cH \rightarrow \cX  ,
 \end{displaymath}
which satisfies Galerkin-orthogonality
\begin{displaymath}
 \langle V(\varphi-\Pi_{\cX}\varphi),v\rangle=0,~~~~ \forall v \in \cX.
\end{displaymath}
We refer to this projection as the Galerkin projection.
For ease of notation we will denote the Galerkin projection onto the space
$\cX_{\ell_x}^x\otimes \cX_{\ell_t}^t$ by $\Pi_{\ell_x,\ell_t}$.
\end{dfn}

Then for isotropic sparse grids we find indices such that
 \begin{equation}\label{eq-combistand}
   \ell_x+\ell_t =  L-l,~l=0,1.
 \end{equation}
 Running over these indices gives the combination technique sparse grid solution $\varphi_L$:
\begin{equation}\label{eq-sgct}
 \varphi_L  = \left(\sum_{l=0}^L\Pi_{l,L-l}\varphi-\sum_{l=0}^L\Pi_{l-1,L-l}\varphi\right)\in \cX_L^\sigma,~\sigma=1. 
\end{equation}
This combination of spaces is shown in Figure \ref{fig-combinationindex}.
Essentially one adds the  spaces
for $l=0$ (denoted by $+$ on the figure) and then subtracts the spaces for $l=1$ 
(denoted by $-$ on the figure).

  If we want to use the combination technique for an anisotropic sparse grid index set, 
  i.e. for a set of the form
 \begin{displaymath}
   \hat I_L^\sigma = \{(\ell_x,\ell_t)~:~\sigma\ell_x+\ell_t/\sigma \leq L\},
 \end{displaymath}
   formula (\ref{eq-combistand}) is changed as follows (see \cite{griebelharbrecht2})
 \begin{displaymath}
   \lceil \sigma^2l_x\rceil+\ell_t = \lceil\sigma L\rceil-l,~l=0,1.
 \end{displaymath}

 \subsubsection{Numerical Experiments}
Next we test the convergence rates derived in Section \ref{sec-sg}, we compare the
convergence rates of the square of the energy norm of the full tensor product discretisation
and the isotropic sparse grid discretisation with a scaling of $\sigma=1$ in both.
As before we chose the time horizon to be $T=4$ and solve the Dirichlet problem (\ref{1-dir})
on a unit circle with the right hand side 
 $g(r,\varphi,t)=t^2\cos(\varphi)$. 

The convergence rate for the full tensor product discretisation, namely $\frac{5}{8}$, is as expected.
The expected convergence rate for the square of the energy norm for the isotropic sparse grid method is $\frac{7}{6}$. The tests
 in Figure \ref{fig-combination2} show a correspondence to the expected rates.

Lastly, we give some numerical results for the combination technique. The left plot in Figure \ref{fig-combination2}
shows convergence of the energy norm against the total number of degrees of freedom.
As expected, the convergence rates are identical to those obtained by implementing the
sparse grid method using a multilevel decomposition. However,  as the right plot in Figure \ref{fig-combination2} shows
the combination technique provides a large improvement in the time taken for the calculation.

\begin{figure}[t] 
 \centering
 \includegraphics{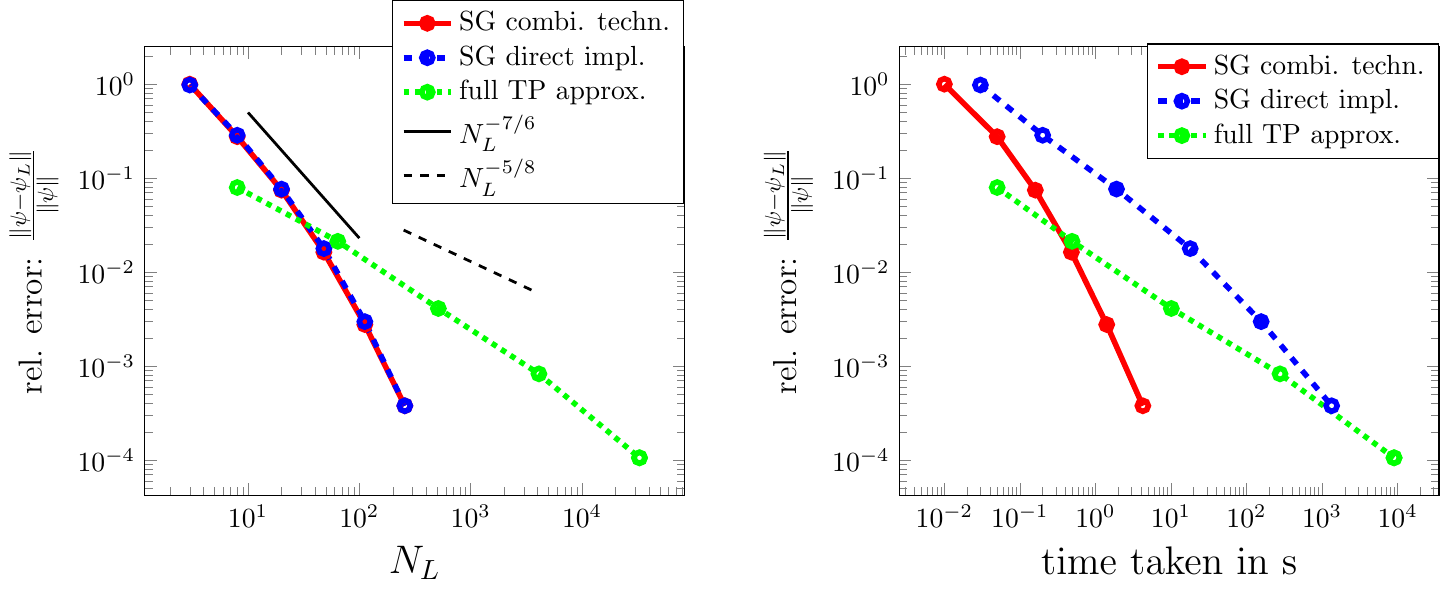}
 \caption{Convergence of the relative error in the energy norm squared versus number of degrees of freedom (left) and time taken in seconds (right) for
the standard sparse grid space with the combination technique and without.}
 \label{fig-combination2}
 \end{figure} 

\subsection{Adaptive Sparse Grids}
  \begin{figure}[t]
  \centering
  \includegraphics{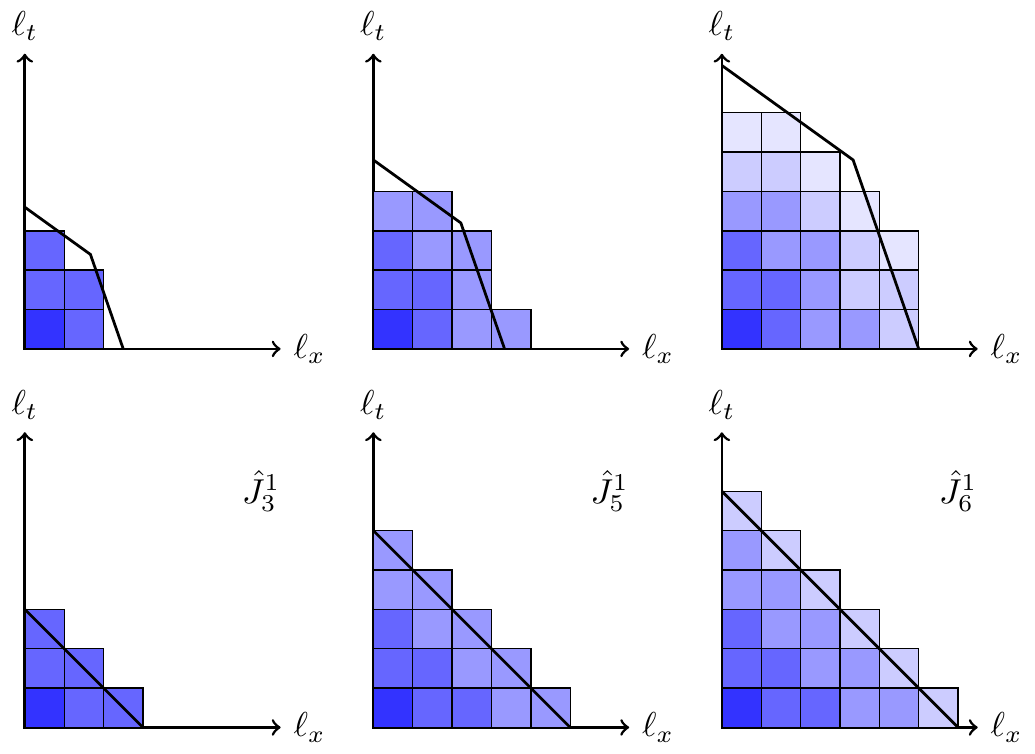}
  \caption{Growth of the adaptive index set shown by squares (top row). Darker squares denote the index set of earlier steps.
  The black line shows the a-priori index set $J_L^{\mathcal{T},\sigma}$.
  This is compared to the growth of the standard anisotropic sparse grid $\hat J_L^\sigma$ (bottom row).}
 \label{fig-adapt}
 \end{figure}

Independently of the choice of approximation space $\hat\cX_L$ we can make the following an Aubin-Nitsche argument.
Let us assume that the solution $\psi$ is in $L^2(\Sigma)$. 
\begin{displaymath}
\begin{split}
 \|\psi-\Pi_{\hat\cX_L} \psi\|_{\cH} 
 & = \sup_{\xi\in H^{\frac{1}{2},\frac{1}{4}}(\Sigma)}\frac{\langle \psi-\Pi_{\hat\cX_L} \psi,\xi\rangle}{\|\xi\|_{ H^{\frac{1}{2},\frac{1}{4}}(\Sigma)}}\\
 &=\sup_{\xi\in  H^{\frac{1}{2},\frac{1}{4}}(\Sigma)} \frac{\langle \psi-\Pi_{\hat\cX_L} \psi,\xi-\Pi_{\hat\cX_L} \xi\rangle}{\|\xi\|_{ H^{\frac{1}{2},\frac{1}{4}}(\Sigma)}}
  \end{split}
 \end{displaymath}
 Then we can estimate
 \begin{displaymath}
\begin{split}
 &\|\psi-\Pi_{\hat\cX_L} \psi\|_{\cH} \leq \|\psi-\Pi_{\hat\cX_L} \psi\|_{L^2(\Sigma)}
 \sup_{\xi\in  H^{\frac{1}{2},\frac{1}{4}}(\Sigma)} \frac{\|\xi-\Pi_{\hat\cX_L} \xi\|_{L^2(\Sigma)}}{\|\xi\|_{ H^{\frac{1}{2},\frac{1}{4}}(\Sigma)}}\\
 &~~~~~\leq \|\psi\|_{H^{s_x,s_t}_\text{mix}(\Sigma)}\underbrace{\frac{\|\psi-\Pi_{\hat\cX_L} \psi\|_{L^2(\Sigma)}}{\|\psi\|_{H^{s_x,s_t}_\text{mix}(\Sigma)}}}
 \underbrace{\sup_{\xi\in \|\psi\|_{H^{\frac{1}{2},\frac{1}{4}}(\Sigma)}}\frac{\|\xi-\Pi_{\hat\cX_L} \xi\|_{L^2(\Sigma)}}{\|\xi\|_{ H^{\frac{1}{2},\frac{1}{4}}(\Sigma)}}}\\
 &~~~~~~~~~~~~~~~~~~~~~~~~~~~~~~\text{small for}~~~~~~~~~~~~~\text{small for full tensor}\\
 &~~~~~~~~~~~~~~~~~~~~~~~~~~~~~~\text{standard}~~~~~~~~~~~~~~~\text{ product grids}\\
 &~~~~~~~~~~~~~~~~~~~~~~~~~~~~~\text{sparse grids}\\
 \end{split}
 \end{displaymath}
 for any approximation space $\hat\cX_L$. This argument leads to the idea of finding 
 a compromise between the full tensor product discretisation and the sparse grid discretisation.
 However, it is unclear what form this index set should take. The optimised index sets given
 in \cite{griebelknapek} are one possibility. However, they still leave us several
 free parameters that need to be chosen.

 In order to get a better idea of how a good index set might look we implemented
 a simple adaptive scheme from \cite{griebelbungartz}. Since our aim was to gain a better understanding of the optimal
 shape of the index set, we have not used an error indicator to choose the next index to include.
 Instead we have simply calculated
 the actual contribution to the energy norm for each index.
 
For an index $(\ell_x,\ell_t)$ the local cost function $c(\ell_x,\ell_t)$ is given by the size of the
corresponding subspace $W_{(\ell_x,\ell_t)}$, i.e.
\begin{displaymath}
 c(\ell_x,\ell_t) = 2^{\ell_t+l_x(d-1)}.
\end{displaymath}
Our local benefit function $b(\ell_x,\ell_t)$ gives the contribution of the index to the energy norm, i.e.
\begin{displaymath}
 b(\ell_x,\ell_t) = \|\psi_{(\ell_x,\ell_t)}\|_{\cH}.
\end{displaymath}
In each step of the adaptive algorithm we choose the next admissible index that maximises
the cost-benefit ratio $c(\ell_x,\ell_t)/b(\ell_x,\ell_t)$.
 We require that an admissible index has no ``holes'', i.e. that all downset neighbours are already contained in the
 index set.
 
 Figure \ref{fig-adapt} shows the growth of the adaptively chosen index set. We notice that these
 index sets behave similarly to the optimised index sets from \cite{griebelknapek}
 for negative choices of $\mathcal{T}$. For our purposes we require an anisotropy 
 in the definition of the index sets.
Care has to be taken when comparing to \cite{griebelknapek}, where this anisotropy is not
present in the definition.
\begin{dfn}
The optimised sparse grid index set is defined as follows
\begin{displaymath}
 J_L^{\cT,\sigma}=\left\{(\ell_x,\ell_t)~:~\sigma\ell_x+l_t/\sigma-\cT\max\{\sigma
\ell_x,\ell_t/\sigma\} \leq (1-\cT)L\right\}.
\end{displaymath}
where $\cT\in [-\infty,1)$and $\sigma$ are free variable .
\end{dfn}
The lines in Figure \ref{fig-adapt} show the boundary of the index set $J_L^{\cT,\sigma}$ for
a scaling of $\frac{6}{5}$ and $\cT = -1$ for comparison.

\subsubsection{Numerical Experiments}
\begin{figure}[t] 
\centering
\includegraphics{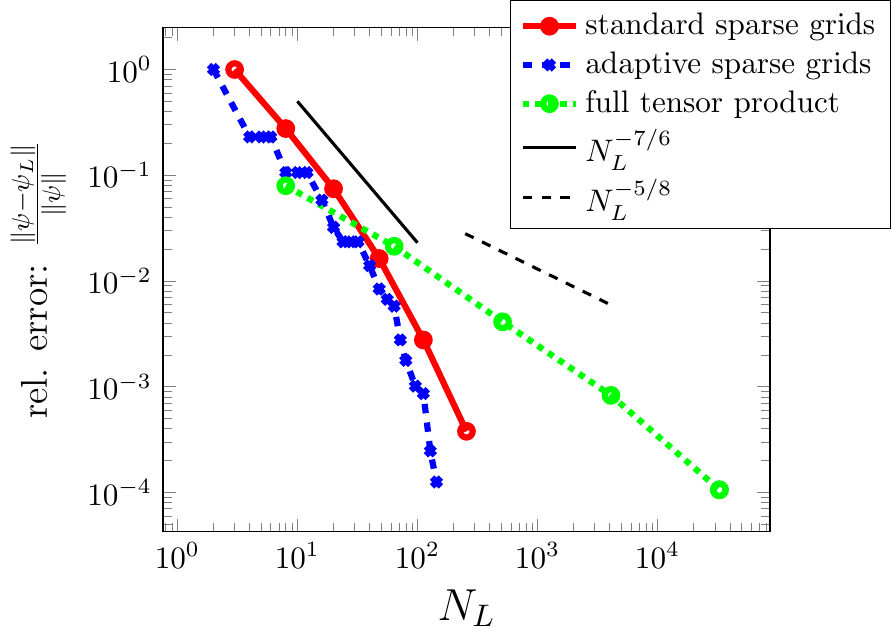}
 \caption{Convergence of the relative error in the energy norm squared versus number of degrees of
 freedom  for the standard sparse grid space and for the adaptively chosen sparse grid space.}
 \label{fig-adapt2}
 \end{figure} 
 
 Next we numerically verify that the adaptively chosen index sets indeed yield an improvement over the standard
 index sets.  In Figure \ref{fig-adapt2} we see that this is the case.
As before we chose the time horizon to be $T=4$ and solve the Dirichlet problem (\ref{1-dir})
on a unit circle with the right hand side 
 $g(r,\varphi,t)=t^2\cos(\varphi)$. For this problem the exact solution is known, so the 
 exact error can be calculated. 
 
 \section{Conclusions}
 The boundary integral formulation efficiently approximates the Cauchy data of the
 given boundary value problem, i.e. the missing boundary values for the Neumann problem
 or the missing boundary fluxes for the Dirichlet problem.
  We have shown an error analysis for two types of
 discretisation spaces. Firstly, we have shown improved error bounds for a full tensor product 
 discretisation.  These improvements are largest for lower
 polynomial degrees, which are primarily of interest in an implementation.
 
 Secondly, we have shown similar improvements for anisotropic sparse grid discretisation
 spaces. In terms of degrees of freedom the sparse grid discretisation
 gives a considerably higher convergence rates than the full tensor product discretisation.
 However,  the sparse grid combination technique is essential to ensure 
 that these gains lead to a corresponding speed-up in run time.
 
 There is scope to explore the  choice of sparse grid index set for this
 problem, as the adaptive algorithm shows that the anisotropic sparse grid index sets
 are not the optimal choice.

\bibliographystyle{elsarticle-num} 
\bibliography{heatBEM}

\begin{thebibliography}{10}
\expandafter\ifx\csname url\endcsname\relax
  \def\url#1{\texttt{#1}}\fi
\expandafter\ifx\csname urlprefix\endcsname\relax\def\urlprefix{URL }\fi
\expandafter\ifx\csname href\endcsname\relax
  \def\href#1#2{#2} \def\path#1{#1}\fi

\bibitem{wayland}
H.~Wayland, Differential equations applied in science and engineering, D. Van
  Nostrand Co., Inc., Princeton, N. J.-Toronto-New York-London, 1957.

\bibitem{wilmott}
P.~Wilmott, S.~Howison, J.~Dewynne, The Mathematics of Financial Derivatives,
  Cambridge University Press, 1995, cambridge Books Online.

\bibitem{witkin}
A.~Witkin, Scale-space filtering: A new approach to multi-scale description,
  in: Acoustics, Speech, and Signal Processing, {IEEE International Conference
  on ICASSP \'84.}, Vol.~9, 1984, pp. 150--153.

\bibitem{thomee}
V.~Thomee, Galerkin finite element methods for parabolic problems, Springer,
  1984.

\bibitem{oeltz}
M.~Griebel, D.~Oeltz, A sparse grid space-time discretization scheme for
  parabolic problems, Computing 81~(1) (2007) 1--34.

\bibitem{costabelheat}
M.~Costabel, Boundary integral operators for the heat equation, Integral
  Equations Operator Theory 13~(4) (1990) 498--552.

\bibitem{chernovschwabheat}
A.~Chernov, C.~Schwab, Sparse space-time {G}alerkin {BEM} for the nonstationary
  heat equation, ZAMM Z. Angew. Math. Mech. 93 (2013) 403--413.

\bibitem{chernovschwabstoch}
A.~Chernov, C.~Schwab, First order $k$-th moment {F}inite {E}lement analysis of
  nonlinear operator equations with stochastic data, Mathematics of Computation
  82 (2013) 1859--1888.

\bibitem{zengergriebel}
M.~Griebel, M.~Schneider, C.~Zenger, A combination technique for the solution
  of sparse grid problems, in: P.~de~Groen, R.~Beauwens (Eds.), {Iterative
  Methods in Linear Algebra}, IMACS, Elsevier, North Holland, 1992, pp.
  263--281.

\bibitem{garckegriebel}
J.~Garcke, M.~Griebel, On the computation of the eigenproblems of hydrogen and
  helium in strong magnetic and electric fields with the sparse grid
  combination technique, Journal of Computational Physics 165~(2) (2000) 694 --
  716.

\bibitem{griebelknapek}
M.~Griebel, S.~Knapek, Optimized general sparse grid approximation spaces for
  operator equations, Math. Comp. 78~(268) (2009) 2223--2257.

\bibitem{noondiss}
P.~Noon, The single layer heat potential and {G}alerkin boundary element
  methods for the heat equation, Phd thesis.

\bibitem{griebelharbrecht}
M.~Griebel, H.~Harbrecht, On the construction of sparse tensor product spaces,
  Mathematics of Computations 82~(282) (2013) 975--994.

\bibitem{owndiss}
A.~Reinarz, Sparse space-time boundary element methods for the heat equation,
  Ph.D. thesis, University of Reading (2015).

\bibitem{zenger}
C.~Zenger, Sparse grids, Notes Numer. Fluid Mech. 31 (1991) 241--251.

\bibitem{griebelbungartz}
H.~Bungartz, M.~Griebel, Sparse grids, Acta Numerica 13 (2004) 1--123.

\bibitem{griebelparallel1}
M.~Griebel, The combination technique for the sparse grid solution of {PDE}s on
  multiprocessor machines, Parallel Processing Letters 2~(1) (1992) 61--70.

\bibitem{griebelparallel2}
M.~Griebel, W.~Huber, T.~St\"ortkuhl, C.~Zenger, On the parallel solution of
  {3D} {PDE}s on a network of workstations and on vector computers, in:
  {Lecture Notes in Computer Science 732, Parallel Computer Architectures:
  Theory, Hardware, Software, Applications}, Springer Verlag, 1993, pp.
  276--291.

\bibitem{griebelharbrecht2}
M.~Griebel, H.~Harbrecht, On the convergence of the combination technique, in:
  Sparse grids and Applications, Vol.~97 of Lecture Notes in Computational
  Science and Engineering, Springer, 2014, pp. 55--74.

\end{thebibliography}

\end{document}